\documentclass[11pt]{amsart}
\usepackage{amssymb}
\usepackage{amsmath}
\usepackage{amsfonts}
\usepackage{graphicx}

\usepackage[total={17cm,22cm},top=2.5cm, left=2.3cm]{geometry}
\usepackage{hyperref}
    \usepackage{aeguill}
    \usepackage{type1cm}

\theoremstyle{plain}
\newtheorem{thm}{Theorem}[section]

\newtheorem{claim}[thm]{Claim}

\newtheorem{definition}[thm]{Definition}

\newtheorem{problem}[thm]{Problem}

\newtheorem{remark}[thm]{Remark}

\newtheorem{theorem}[thm]{Theorem}
\newtheorem{fact}[thm]{Fact}

\newtheorem{ex}[thm]{Example}

\numberwithin{equation}{section}
\newcommand{\N}{\mathbb{N}}

\newcommand{\R}{\mathbb{R}}

\DeclareMathOperator{\lip}{Lip\,\!}

\DeclareMathOperator{\dist}{dist\,\!}

\DeclareMathOperator{\co}{co\,\!}

\DeclareMathOperator{\interior}{int\,\!}

\usepackage[usenames,dvipsnames]{color}

\usepackage[dvipsnames]{xcolor}

\begin{document}

\title[Prescribing tangent hyperplanes]{Prescribing tangent hyperplanes to $C^{1,1}$ and $C^{1, \omega}$ convex hypersurfaces in Hilbert and superreflexive Banach spaces}

\author{Daniel Azagra}
\address{ICMAT (CSIC-UAM-UC3-UCM), Departamento de An{\'a}lisis Matem{\'a}tico y Matem\'atica Aplicada,
Facultad Ciencias Matem{\'a}ticas, Universidad Complutense, 28040, Madrid, Spain.  {\bf DISCLAIMER:} 
The first-named author is affiliated with Universidad Complutense de Madrid, but this does not mean this institution has offered him all the support he expected. On the contrary, the Biblioteca Complutense has hampered his research by restricting his access to books.
}
\email{azagra@mat.ucm.es}

\author{Carlos Mudarra}
\address{ICMAT (CSIC-UAM-UC3-UCM), Calle Nicol\'as Cabrera 13-15.
28049 Madrid, Spain}
\email{carlos.mudarra@icmat.es}

\date{June 25, 2018}



\begin{abstract}
Let $X$ denote $\R^n$ or, more generally, a Hilbert space. Given an arbitrary subset $C$ of $X$ and a collection $\mathcal{H}$ of affine hyperplanes of $X$ such that every $H\in\mathcal{H}$ passes through some point $x_{H}\in C$, and $C=\{x_H : H\in\mathcal{H}\}$, what conditions are necessary and sufficient for the existence of a $C^{1,1}$ convex hypersurface $S$ in $X$ such that $H$ is tangent to $S$ at $x_H$ for every $H\in\mathcal{H}$? In this paper we give an answer to this question. We also provide solutions to similar problems for convex hypersurfaces of class $C^{1, \omega}$ in Hilbert spaces, and for convex hypersurfaces of class $C^{1, \alpha}$ in superreflexive Banach spaces having equivalent norms with moduli of smoothness of power type $1+\alpha$, $\alpha\in (0, 1]$.
\end{abstract}

\maketitle

\section{Introduction}

This paper concerns the following question. 

\begin{problem}\label{main problem}
{\em Let $X$ denote the space $\R^n$ or, more generally, a Banach space, and let $\mathcal{C}$ be a differentiability class. Given an arbitrary subset $C$ of $X$ and a collection $\mathcal{H}$ of affine hyperplanes of $X$ such that every $H\in\mathcal{H}$ passes through some point $x_{H}\in C$, and $C=\{x_H : H\in\mathcal{H}\}$,
what conditions on $\mathcal{H}$ are necessary and sufficient for the existence of a {\em convex} hypersurface $S$ of class $\mathcal{C}$ in $X$ such that $H$ is tangent to $S$ at $x_H$ for every $H\in\mathcal{H}$? }
\end{problem}

In \cite{GhomiJDG2001} M. Ghomi considered a version of this problem and solved it in the particular case that $S$ is an ovaloid (that is to say, a closed hypersurface of strictly positive curvature), $\mathcal{C}=C^{m}$, $m\geq 2$, and $C$ is a $C^m$ smooth submanifold of $\R^n$.

In \cite{AzagraMudarra2017PLMS}, as a consequence of a Whitney-type extension theorem for convex functions of the class $C^{1,1}$, we solved this problem in the case that $C$ is a compact subset of $\R^n$ and $\mathcal{C}=C^{1,1}$. A similar result was given in \cite[Corollary 1.5]{AMHilbert} for $X$ a Hilbert space, $C$ arbitrary, and $\mathcal{C}=C^{1,1}$; however, the proof of this corollary was incomplete in the case that $C$ is unbounded (we only sketched the proof of the {\em easy} implication, and overlooked one important difference between bounded and unbounded convex bodies). 

More recently, in \cite[Corollary 1.15]{AzagraMudarraGlobalGeometry}, we provided a solution to Problem \ref{main problem}, for arbitrary $C$, in the case that $X=\R^n$ and $\mathcal{C}=C^1$.

As far as we know, nothing is known about the case that $C$ is arbitrary and $\mathcal{C}=C^{m}$, $m\geq 2$, and in fact Problem \ref{main problem} looks extremely hard to solve in this generality. The main reasons why we consider this problem very difficult are the facts that: 1) partitions of unity cannot be used to patch local convex extensions, as they  destroy convexity; and 2) convex envelopes do not preserve smoothness of orders higher than $C^{1,1}$.

The aim of this paper is twofold. On the one hand, we wish to clarify what can be understood (and also what we think should be understood, if we want to be practical) by a convex hypersurface of class $C^{1,1}$ or $C^{1, \omega}$ (where $\omega$ is a modulus of continuity) in a Hilbert space, or more generally, in a Banach space. This question will keep us busy in Section \ref{sectionwhatisconvexhypersurfacec11}. On the other hand, we want to give a complete proof of \cite[Corollary 1.5]{AMHilbert}, and furthermore to extend this result to the class $C^{1, \omega}$ and to other Banach spaces. That is, we mean to provide a complete solution to Problem \ref{main problem} for $\mathcal{C}=C^{1, \omega}$. 

Of course, the solution to Problem \ref{main problem} will depend on the notion of convex hypersurface of class $\mathcal{C}$ with which we choose to work. However, as we will see in Section \ref{sectionwhatisconvexhypersurfacec11}, all reasonable notions of $C^{1,1}$ smoothness for a convex hypersurface in a Hilbert space are equivalent, and therefore we can give a precise statement of our main result in this particular case right now. Let us first notice that an equivalent reformulation of Problem \ref{main problem} is the following.

\begin{problem}\label{main problem 2}
Let $X$ be a Banach space, and let $\mathcal{C}$ be a differentiability class. Denote by $S_X$ the unit sphere of $X$.
Given a subset $C$ of $X$ and a mapping $N: C \to S_X$, what conditions are necessary and sufficient to ensure the existence of a (not necessarily bounded) convex body $V$ of class $\mathcal{C}$ such that $C \subseteq \partial V$ and the outer unit normal to $\partial V$ coincides  on $C$ with the given mapping $N$?
\end{problem}

One of the main results of this paper is the following.

\begin{theorem}\label{main theorem for C11 Hilbert}
Let $C$ be a subset of a Hilbert space $X$, and let $N:C\to S_X$ be a mapping. Then the following statements are equivalent.
\begin{enumerate}
\item There exists a $C^{1,1}$ convex body $V$ such that $C\subseteq \partial V$ and $N(x)$ is outwardly normal to $\partial V$ at $x$ for every $x\in C$.
\item There exists some $r>0$ such that
$$
\langle N(y), y-x\rangle \geq \tfrac{r}{2} \|N(y)-N(x)\|^2  \quad \textrm{for all} \quad x, y\in C.
$$
\end{enumerate}
Moreover, if $(2)$ is satisfied with a constant $r>0$, then the body $V$ in $(1)$ can be taken so that the outward unit normal $N_{\partial V}:\partial V\to S_X$ is $r^{-1}$-Lipschitz.

In addition, if we further assume that $C$ is bounded then $V$ can be taken to be bounded as well.
\end{theorem} 

An equivalent reformulation of this result which was suggested to us by Arie Israel is the following {\em finiteness principle} for Problem \ref{main problem}.

\begin{theorem}\label{main theorem for C11 Hilbert, finiteness principle version}
Let $C$ be a subset of a Hilbert space $X$, and let $\mathcal{H}$ be a collection of affine hyperplanes of $X$ such that every $H\in\mathcal{H}$ passes through some point $x_H\in C$, and $C=\{x_H : H\in\mathcal{H}\}$. The following statements are equivalent:
\begin{enumerate}
\item There exists a convex hypersurface $S$ of class $C^{1,1}$ in $X$ such that $S$ has bounded principal curvatures and $H$ is tangent to $S$ at $x_H$ for every $H\in\mathcal{H}$.
\item There exists some $M>0$ such that, for every couple $H_1, H_2$ of hyperplanes in $\mathcal{H}$, there exists a convex hypersurface $S(H_1, H_2)$ of class $C^{1,1}$ such that the principal curvatures of $S(H_1, H_2)$ are bounded by $M$ and $S(H_1, H_2)$ is tangent to $H_1$ and $H_2$ at $x_{H_1}$ and $x_{H_2}$, respectively.
\end{enumerate}
\end{theorem}
By saying that the principal curvatures of a $C^{1,1}$ convex body $V$ are bounded by some constant $M\geq 0$ we mean that the Lipschitz constant of the Gauss map $N_{\partial V}:\partial V\to S_X$ is bounded by $M$. This terminology is natural enough, since for a $C^2$ convex body $W$ the principal curvatures of $\partial W$ at a point $x\in\partial W$ are the eigenvalues of the differential of $N_{\partial W}$ at $x$, and these eigenvalues are bounded by $\textrm{Lip}(N_{\partial W})$. In our setting $N_{\partial V}$ may not be differentiable at some points (but it will be almost everywhere differentiable if $X=\R^n$, thanks to Rademacher's theorem).

For the precise statements of our main results in the cases that $\mathcal{C}=C^{1, \omega}$, where $\omega$ is a modulus of continuity, or that $X$ is a superreflexive Banach space, see Section \ref{sectionmainresults}.


\section{What is a convex hypersurface of class $C^{1,1}$?}\label{sectionwhatisconvexhypersurfacec11}

There is no controversy about what a convex hypersurface is. At least in the case $X=\R^n$, the following definition is accepted (explicitly or implicitly) everywhere in the literature.

\begin{definition}
{\em Let $X$ be a Banach space, and $S$ be a subset of $X$. We will say that $X$ is
a convex hypersurface in $X$ provided that there exists a closed convex set $V$ with nonempty interior such that $S=\partial V$.
}
\end{definition}
Such a set $V$ is sometimes called a {\em convex body}. However, some authors prefer to call such a set $V$ a convex body only if $V$ is bounded too. That is, for some authors a convex body is always bounded, while others indulge in dealing with {\em unbounded} convex bodies as well. {\bf In this paper, convex bodies are allowed to be unbounded.}

As for the $C^m$ regularity ($m\in\N$) of a convex hypersurface $S=\partial V$, there can be no dispute either. The most natural definition from a geometrical point of view is that $S$ be a submanifold of $\R^n$ of class $C^m$, and this happens to be equivalent to one of the most practical analytical definitions, namely, that the Minkowski functional (or gauge) of $V$, denoted by $\mu_V$, be of class $C^m$ on $X\setminus\mu_{V}^{-1}(0)$. Recall that, given a convex body $V$ in $X$, up to a translation we may always assume that $0\in\textrm{int}(V)$, and then define the Minkowski functional of $V$ by
$$
\mu_V(x)=\inf \left\lbrace t>0 : \frac{1}{t}x\in V \right\rbrace.
$$

However, for the regularity class $C^{1,1}$, these two definitions are no longer equivalent and, what is worse, in the literature there seems to be no universal agreement about what a hypersurface of class $C^{1,1}$ is. Some authors say that a submanifold $M$ of $\R^n$ is of class $C^{1,1}$ provided that $M$ is {\em locally} of class $C^{1,1}$ (meaning, for instance, that a normal to $M$ is locally Lipschitz), while other authors demand that a uniform Lipschitz constant should exist. These definitions are equivalent in the case that $V$ is bounded. Therefore, there cannot be any disagreement, either, about what a {\em bounded} convex body of class $C^{1,1}$ is.

We are thus left with the following questions: what is an {\em unbounded} convex body of class $C^{1,1}$? And, more generally, what is a (not necessarily convex) hypersurface of class $C^{1,1}$?

In this paper we will make a distinction between hypersurfaces of class $C^{1,1}$ and hypersurfaces of class $C^{1,1}_{\textrm{loc}}$.

\begin{definition}\label{definition of C11 hypersurface}
{\em Let $M$ be a $C^1$ hypersurface of $\R^n$ or, more generally, of a Hilbert space $(X, \|\cdot\|)$, and let $N:M\to S_X$ be a continuous unit normal. We will say that $M$ is of class $C^{1,1}$ provided that $N$ is Lipschitz continuous (with respect to the ambient distance), that is to say, there exists some constant $L>0$ such that
$$
\|N(x)-N(y)\|\leq L\|x-y\|
$$
for all $x, y\in M$.

We will say that $M$ is of class $C^{1,1}_{\textrm{loc}}$ whenever this condition holds locally, that is, for every $z\in M$ there exists some positive numbers $r, L$, depending on $z$, such that
$$
\|N(x)-N(y)\|\leq L\|x-y\|
$$
for all $x, y\in B(z, r)$.}
\end{definition} 

We will use a similar terminology for functions. A function $F:X\to\R$ will be said to be of class $C^{1,1}$ provided $F\in C^{1}(X)$ and the gradient $\nabla F$ is (globally) Lipschitz. If $\nabla F$ is locally Lipschitz then we will say that $F$ is of class $C^{1,1}_{\textrm{loc}}$.

\medskip

In $\R^n$ it is well known that a convex hypersurface $M=\partial V$ is of class $C^{1,1}$ if and only if there exists an $r>0$ such that the balls of radii $r$ inside $V$ roll freely on $M$. This means that for every $x\in M$ there exists a ball $B(z,r)\subset V$ such that  $\partial B(z,r)\cap M=\{x\}$; of course in this case we have $z=x- rN(z)$, where $N:\partial V\to S_X$ is the outer unit normal;  see \cite{Lucas, GhomiHoward} and the references therein for instance. It is also known, see \cite{DelfourZolesio, KrantzParks}, that if $M=\partial A$ is the boundary of a proper open subset $A$ of $\R^n$, then $M$ is of class $C^{1,1}$ if and only if there exists some $r>0$ such that the signed distance to $M$, defined by
$$
b_A(x)  =  \left\lbrace
	\begin{array}{ccl}
	d(x,A) & \mbox{ if } & x\in X \setminus A\\
	0  & \mbox{ if } & x\in \partial A \\
	- d(x,\partial A) & \mbox{ if } & x\in \interior(A),
	\end{array}
	\right. 
$$
is of class $C^{1,1}$ on the set $\{x\in\R^n : \textrm{dist}(x, \partial A)<r\}$. In the case that $A$ is convex, this fact allows us to realize $M$ as a level set of a $C^{1,1}$ convex function defined on $\R^n$. This kind of representation becomes very useful when we want to transfer smooth extension results from convex functions to convex bodies, and the other way around.

Unfortunately, the proofs of these finite-dimensional results do not immediately extend to Hilbert spaces, mainly due to the following fact: if $A$ is an open {\em convex} subset of an infinite-dimensional Hilbert space and $x\in \textrm{int}(A)$, the distance of $x$ to $\partial A$ may not be attained. An instance of this situation is provided by the following.
\begin{ex}
{\em Let $X$ be a separable Hilbert space of infinite dimension, and let us denote an orthonormal basis of $X$ by $\{e_n\}_{n\in\N}$. Define
$$
W=\left\{x\in X \, : \, \sum_{n=1}^{\infty}\frac{\langle e_n, x\rangle ^2}{(1+ 2^{-n})^2}\leq 1\right\}.
$$
Then $W$ is a bounded convex body, and clearly we have that
$$
d(0, \partial W)=1.
$$
However, it is not difficult to see that the closed ball $B(0,1)$ is contained in the interior of $W$, hence this distance is not attained.}
\end{ex}

Our aim in this section is to show that these results remain nonetheless true for convex bodies in Hilbert spaces, thus providing several equivalent definitions of $C^{1,1}$ smoothness for (possibly unbounded) convex bodies. We do not claim that the following two results are completely original. As a matter of fact, some of the properties and implications of these two results are proved, in a more general setting, in the papers \cite{ClarkeSternWolenski, PoliquinRockafellarThibault}, which explore the notion of {\em proximally smooth sets}; for instance see \cite[Theorems 4.1 and 4.8]{ClarkeSternWolenski}, and \cite[Theorem 4.1]{PoliquinRockafellarThibault}. What seems to be new in our results of this section is the equivalence of $(3)$, $(4)$ and $(5)$ of Theorem \ref{characterization of C11 convex bodies} below, and perhaps also $(1)$ of Theorem \ref{regularity of the signed distance} (for which we have been unable to find a reference). We could have limited ourselves to proving what we think is new or could not find in the literature, providing a reference for what is known, but we chose to present a self-contained proof of these results which does not rely on the less elementary notions and tools of the papers \cite{ClarkeSternWolenski, PoliquinRockafellarThibault}.

\begin{theorem}[Regularity of the signed distance to the boundary of a $C^{1,1}$ convex body]\label{regularity of the signed distance}
Let $V$ be a convex body of class $C^{1,1}$ in a Hilbert space $X$. Let us denote the signed distance to $\partial V$ by $b_V$, and the outer unit normal to $S:=\partial V$ by $N_S: S\to S_{X}$. Then the following properties are satisfied:
\begin{enumerate}
\item If $x \in X$ is such that $b_V(x)>- \lip(N_S)^{-1},$ then the distance $\dist(x, S)$ is attained at a unique point, which we will denote by $P_{S}(x),$ and $x-P_S(x) = b_V(x) N_S(P_S(x)).$ 
\item For every $\varepsilon \in (0,1)$, the mapping $P_S: \lbrace z\in X \: :\: b_V(z) \geq -\lip(N_S)^{-1} \rbrace \to S$ satisfies
$$
\langle P_S(x)-P_S(y), x- y \rangle \geq \varepsilon \|  P_S(x)-P_S(y) \|^2 \quad \text{for every} \quad x,y\in U_\varepsilon, \quad \text{where} 
$$
$$
U_\varepsilon:=\lbrace z\in X \: :\: b_V(z) \geq - (1-\varepsilon)\lip(N_S)^{-1} \rbrace.
$$
In particular, $P_S$ is Lipschitz on $U_\varepsilon$ with $\lip(P_S, U_\varepsilon) \leq \frac{1}{\varepsilon}.$ 
\item The function $b_V$ is Fr\'echet differentiable at every point $x\in X$ such that $b_V(x)>- \lip(N_S)^{-1},$ with $\nabla b_V(x)=N_S\left(P_{S}(x) \right).$ In particular, $\nabla b_V= N_S$ on $S=\partial V$. Moreover, $\nabla b_V$ is Lipschitz on each $U_\varepsilon,$ with $\lip(\nabla b_V,U_\varepsilon) \leq \tfrac{1}{\varepsilon} \lip(N_S)$, for every $\varepsilon \in (0,1).$
\item The function $b_V$ is convex on $X.$
\end{enumerate}
\end{theorem}

\begin{theorem}\label{characterization of C11 convex bodies}
Let $S$ be a convex hypersurface of a Hilbert space $X$; say $S=\partial V$, where $V$ is a closed convex body (not necessarily bounded). Assume that $S$ is a $C^1$ submanifold, so that the outer unit normal $N_S:S\to S_X$ is well defined. Then, the following statements are equivalent:
\begin{enumerate}
\item The mapping $N_S :S\to S_{X}$ is $L$-Lipschitz.
\item For every $0<r<1/L$, the balls of radii $r$ inside $V$  {\em roll freely on $S$,} meaning that for every $x\in S$ there exists a ball $B(z,r)\subset V$ such that $\partial B(z,r)\cap S=\{x\}$.
\item The mapping $N_S$ satisfies 
$$
\langle N_S(y), y-x \rangle \geq  \tfrac{1}{2L} \| N_S(x)-N_S(y)\|^2 \quad \text{for every} \quad x,y\in S.
$$
\item There exists a {\em convex} function $F:X\to\R$ of class $C^{1,1}$ such that $\lip(\nabla F) \leq L,$ $S=F^{-1}(1)$ and $\nabla F(x)=N_S(x)$ for every $x\in S$.
\end{enumerate}
Furthermore, if $V$ is bounded and $0\in \interior(V)$, then the above statements are also equivalent to:
\begin{enumerate}
\item[{(5)}] $\mu_V$ is of class $C^{1,1}$ on the set $\{x\in X: \mu_V(x) \geq \alpha\}$ for every $\alpha >0.$
\end{enumerate}
\end{theorem}
Therefore any of these conditions can be taken as the definition of a convex body of class $C^{1,1}$.

In the remainder of this section we will prove Theorems \ref{regularity of the signed distance} and \ref{characterization of C11 convex bodies}. We will first establish $(1)$ of Theorem \ref{regularity of the signed distance}, then we will turn to the proof of Theorem \ref{characterization of C11 convex bodies}, and finally return to proving $(2), \: (3)$ and $(4)$ of Theorem \ref{regularity of the signed distance}.

\medskip

\subsection*{Proof of Theorem \ref{regularity of the signed distance} $(1)$} If $x \in X \setminus \interior(V)$, because $V$ is closed and convex and the norm $\|\cdot\|$ is Hilbertian, we can find a unique point $P_S(x) \in \partial V= S$ such that $\| x-P_S(x)\| = \dist(x,S)$ and the mapping $ X \setminus \interior(V) \ni x \mapsto P_S(x)$ is $1$-Lipschitz. 

Let us now assume that $x_0\in \interior(V)$ with $r:=\dist(x,S)< \lip(N_S)^{-1}.$ We can find a sequence $(z_n)_n $ in $S$ such that $\lim_n \| z_n-x_0\| = r.$ If we define $x_n:= z_n- r N_S(z_n),$ we claim that $(x_n)_n$ converges to $x_0.$ Indeed, if $n$ is large enough so that $r> \frac{1}{n},$ the point $x_0+ (r-\frac{1}{n}) N_S(z_n)$ belongs to the interior of the ball $B(x_0,r)$ and then $x_0+ (r-\frac{1}{n}) N_S(z_n) \in V.$ The convexity of $V$ implies that
$$
\big \langle N_S(z_n), x_0 + (r-\tfrac{1}{n})N_S(z_n) -z_n \big \rangle \leq 0.
$$
This allows us to write
\begin{align*}
&\|x_n-x_0\|^2 = \| z_n-x_0-r N_S(z_n) \|^2 = \| z_n-z_0\|^2 + r^2 \| N_S(z_n) \|^2 + 2r \big \langle N_S(z_n), x_0-z_n \big \rangle \\
& = \| z_n-z_0\|^2 + r^2  + 2r \big \langle N_S(z_n), x_0+(r-\tfrac{1}{n})N_S(z_n)-z_n \big \rangle -2r \big \langle N_S(z_n), (r-\tfrac{1}{n})N_S(z_n) \big \rangle \\
& \leq \| z_n-z_0\|^2 + r^2  -2r \big \langle N_S(z_n), (r-\tfrac{1}{n})N_S(z_n) \big \rangle = \| z_n-z_0\|^2 + r^2  -2r(r-\tfrac{1}{n}).
\end{align*}
The last term tends to $r^2+r^2-2r^2=0$ as $n \to \infty.$ This shows that $\lim_n \| x_n-x_0\| =0.$ Now, since $N_S$ is Lipschitz we can write, for every $n,m\in \N,$
$$
\|z_n-z_m\| = \| x_n + r N_S(z_n)-x_m + r N_S(z_m) \| \leq \| x_n-x_m\| + r \lip(N_S) \| z_n-z_m\|.
$$
This leads us to
$$
(1- r\lip(N_S)) \| z_n-z_m\| \leq \| x_n-x_m\|, \quad n,m \in \N,
$$
which shows that $(z_n)_n$ is a Cauchy sequence because so is $(x_n)_n$ and $r < \lip(N_S)^{-1}.$ Thus there exists some $z_0 \in S$ with $\dist(x,S)= \lim_n \| z_n-x\| = \| z_0-x\|.$ This proves that the distance function to $S$ is attained on the set $\lbrace z\in X \: :\: \dist(z,S) < \lip(N_S)^{-1} \rbrace.$ In addition, bearing in mind that $S$ is a one-codimensional manifold of class $C^1$ and $N_S$ is the outer unit normal to $S,$ it is straightforward to see that, for every $x\in X$ and $y\in S:$
\begin{equation}\label{characterizationminimizingpoints}
\| x-y\| = \dist(x,S) \quad \text{if and only if} \quad x-y= b_V(x) N_S(y).
\end{equation}
Let $x$ be a point with $\dist(x,S) < \lip(N_S)^{-1},$ or equivalently $|b_V(x)| <\lip(N_S)^{-1}$, and let us see that $\dist(x,S)$ is attained at a unique point. We already know that the distance $\dist(x,S)$ is attained at some $y\in S.$ Assume that there are different points $y_1, y_2 \in S$ such that $\dist(x,S)= \| x-y_1\|= \|x-y_2\|.$ It then follows from \eqref{characterizationminimizingpoints} that
$$
x-y_1 = b_V(x) N_S(y_1), \quad x-y_2= b_V(x)N_S(y_2);
$$
which easily implies that
$$
\| y_1-y_2 \| \leq |b_V(x)| \| N_S(y_2)-N_S(y_1)  \| \leq |b_V(x)| \lip(N_S) \| y_1-y_2 \| < \| y_1-y_2\|,
$$
a contradiction. Therefore, the point $y$ is the unique $y$ for which we have $\| x-y\| =\dist(x,S).$ 

\subsection*{Proof of Theorem \ref{characterization of C11 convex bodies}}

\noindent $(1) \implies (2):$ Let $x\in S$ and $r \in (0, 1/L).$ We first claim that $z:=x-rN_S(x) \in \interior(V).$ Indeed, otherwise we would have $x-tN_S(x) \in S$ for some $t \in (0,r]$ and by convexity of $V$
$$
0 \leq t^{-1}\langle N_S(x-tN_S(x)), x-t N_S(x)-x \rangle = \tfrac{1}{2}\| N_S(x)-N_S(x-tN_S(x))\|^2 -1 \leq \tfrac{L^2}{2}t^2-1,
$$
which is absurd. Thus $z\in \interior(V).$ Now assume for the sake of contradiction that $B(z,r)$ is not contained in $V,$ where $z=x-rN_S(x).$ We have that $t:=\dist(z,S)<r$ and by Theorem \ref{regularity of the signed distance} $(1)$ there exists a unique $y\in S$ such that $\|z-y\| = t.$ Moreover, by the characterization \eqref{characterizationminimizingpoints}, $y$ satisfies $y=z+tN_S(y)$ and we can write
$$
rN_S(x)-rN_S(y) = x-z-tN_S(y) + (t-r)N_S(y) = x-y+(t-r)N_S(y).
$$
By convexity of $V$ we have $\langle N_S(y), y-x \rangle \geq 0$ and then
$$
r^2 \| N_S(x)-N_S(y)\|^2 = \| x-y\|^2 +(r-t)^2 + 2(r-t) \langle N_S(y), y-x \rangle \geq \| x-y\|^2+(r-t)^2>\|x-y\|^2.
$$
This is a contradiction because $N_S$ is $L$-Lipschitz. We have shown that $B(z,r) \subset V$. 

Finally, if $y\in B(z,r) \cap S,$ then $\| y-z\|  \leq r = \dist(z,S)$ as $B(z,r) \subset V.$ This proves that $y=x$ because the distance $\dist(z,S)$ is attained at a unique point.

\medskip

\noindent $(2) \implies (3):$ Given $0<r<1/L$ and $x\in S,$ there exists a ball $B(z_x,r)$ contained in $V$ and such that $B(z_x,r) \cap S= \lbrace x \rbrace.$ The tangent hyperplane to $S$ at the point $x$ coincides with the tangent hyperplane to $\partial B(z_x,r)$ at $x,$ and this implies that $N_S(x) = (x-z_x) / \| x-z_x\|.$ Hence $z_x= x-r N_S(x).$

Now we consider two points $x,y\in S$ and define $p:= x+r(N_S(y)-N_S(x)).$ It is immediate that $p\in B(x-rN_S(x),r) \subset V$ and then $\langle N_S(y),y-p \rangle \geq 0$ since $V$ is convex. Consequently we have
\begin{align*}
 \langle N_S(y), y-x \rangle \geq  \langle N_S(y), p-x \rangle = r \langle N_S(y), N_S(y)-N_S(x) \rangle = \tfrac{r}{2}\| N_S(x)-N_S(y) \|^2
\end{align*}
Since $r\in (0, 1/L)$ is arbitrary we obtain the desired inequality.

\medskip

\noindent $(3) \implies (4):$ If $V$ is a half-space (that is, $S$ is a hyperplane) then the result is obvious. Therefore we may assume that $V$ is not a half-space. Let us consider the $1$-jet on $S$ given by $(f,G)=(1, N_S).$ It is immediate from $(3)$ that
$$
f(x)-f(y)-\langle G(y),x-y \rangle \geq \tfrac{1}{2L} \| G(x)-G(y)\|^2 \quad \text{for all} \quad x,y\in S.
$$
Thus we can apply \cite[Theorem 2.4]{AzagraMudarraExplicitFormulas} to obtain a convex function $F \in C^{1,1}(X)$ such that $(F, \nabla F)=(f,G)= (1,N_S)$ on $S$ and $\lip(\nabla F) \leq L.$ Let us see that, in fact, $F^{-1}(1)=S.$ Indeed, there exist points $x$ such that $F(x)<1$ as otherwise $\nabla F=0$ on $S.$ Thus $W:=F^{-1}(-\infty,1] $ is a closed convex body such that $S\subseteq \partial W=F^{-1}(1)$. Also, given any $x\in X \setminus V,$ the convexity of $F$ together with the fact that $\nabla F=N_S$ give us
$$
F(x) \geq F(P_V(x)) + \langle N_S(P_V(x)), x-P_V(x) \rangle = 1 + \dist(x,V)>1,
$$
where $P_V(x)$ denotes the projection of $x$ onto $V.$ This shows that $W \subseteq V.$ 

To show that $V\subseteq W$, we need to use the following.
\begin{fact}\label{if V is not a half-space}
{\em
If a convex body $V$ is not a half-space, and if $x_0\in\textrm{int}(V)$, then there exists a direction $v\in X \setminus\{0\}$ such that the line $\mathcal{L}:=\{x_0+tv: t\in\R\}$ intersects $S=\partial V$ at exactly two points $x_1=x_0+t_1v$ and $x_2=x_0+t_2v$, with $t_1<0<t_2$.}
\end{fact} 
Assuming this is true for a moment, let us see why $V\subseteq W$.
Assume there exists $x_0\in V$ such that $x_0\notin W$, that is, $F(x_0)>1$. Since $F=1$ on $\partial V$ we necessarily have $x_0\in\textrm{int}(V)$. Let $v, x_1, x_2, t_1, t_2$ be as in Fact \ref{if V is not a half-space}. Then the convex function $\varphi:\R\to\R$ defined by $\varphi(t)=F(x_0+tv)$ takes the value $1$ at the points $t_1$ and $t_2$, while $\varphi(0)>1,$ which is absurd. Therefore we must have $V=W$, and consequently $S=F^{-1}(1)$.

\medskip

Now let us prove Fact \ref{if V is not a half-space}. By assumption $V$ admits at least two different support hyperplanes, say $\mathcal{H}_1$, $\mathcal{H}_2$, the boundaries of two open half-spaces $\mathcal{U}_1, \mathcal{U}_2$, both containing $V$. Then there exists $v\in X \setminus\{0\}$ such that the line $\mathcal{L}:=\{x_0+tv: t\in\R\}$ intersects $\mathcal{H}_1$ and $\mathcal{H}_2$ at two different points $y_1\in\mathcal{H}_1$, $y_2\in\mathcal{H}_2$. We may write $y_1=x_0+s_1v$, $y_2=x_0+s_2v$, and assume (up to replacing $v$ with $-v$ if necessary) that $s_1<0<s_2$. Since $x_0\in\textrm{int}(V)\subseteq\mathcal{U}_1\cap\mathcal{U}_2$, $V$ is a convex body, $\mathcal{H}_1$ and $\mathcal{H}_2$ support $V$, the ray $\{x_0+tv :t<0\}$ intersects $\mathcal{H}_1$, and the ray $\{x_0+tv: t>0\}$ intersects $\mathcal{H}_2$, we may conclude that there exist unique numbers $t_1\in [s_1, 0)$ and $t_2\in (0, s_2]$ such that $x_0+t_1v\in\partial V$ and $x_0+t_2v\in\partial V$.

\medskip

\noindent $(4) \implies (1):$ It is immediate since $\nabla F$ is $L$-Lipschitz.

\medskip

Let us now further assume that $V$ is bounded and $0 \in \interior(V).$ We have that $\mu_V^{-1}(0)=0$ and, since $V$ is a convex body of class $C^1,$ we know that $\mu_V$ is differentiable on $X \setminus \lbrace 0 \rbrace$ and
\begin{equation}\label{comparisongradientminkowskinormal}
\nabla \mu_V(x) = \frac{1}{ \Big \langle  N_S \left( \frac{x}{\mu_V(x)} \right), \frac{x}{\mu_V(x)} \Big \rangle }  N_S\left( \frac{x}{\mu_V(x)} \right) \quad \text{for all} \quad x\in X \setminus \lbrace 0 \rbrace.
\end{equation}
In particular, we have that
\begin{equation}\label{gradientminkowskifunctional}
\langle \nabla \mu_V(z), z \rangle =1 \quad \text{for all} \quad z\in S.
\end{equation}
Let $0<r \leq R$ be such that 
\begin{equation}\label{ballscontaininbody}
B(0,r) \subset V \subset B(0,R).
\end{equation}
Then $\mu_V$ is $r^{-1}$-Lipschitz and $\mu_V \geq R^{-1} \|\cdot \|$ on $X.$ Therefore $\|\nabla \mu_V\| \leq r^{-1}$ and also, because $N_S= \nabla \mu_V/ \| \nabla \mu_V\|$ on $S,$ the identity \eqref{gradientminkowskifunctional} gives
\begin{equation} \label{geometricestimationc11normal}
\langle N_S(z), z \rangle \geq r \quad \text{for all} \quad z\in S.
\end{equation}
Finally, combining \eqref{gradientminkowskifunctional} with \eqref{ballscontaininbody} we obtain
\begin{equation}\label{infimumminkowskigradient}
\| \nabla \mu_V (z)\| \geq R^{-1} \quad \text{for all} \quad z\in S.
\end{equation}

Now let us see why $(1)$ and $(5)$ are equivalent. 

\noindent $(1) \implies (5):$ Let us assume that $N_S: S \to S_X$ is Lipschitz. Using that $\mu_V$ is $r^{-1}$-Lipschitz and \eqref{ballscontaininbody} we can write, for every $x,y\in X\setminus \lbrace 0 \rbrace,$
\begin{align}\label{previousestimationinverseminkowski}
\bigg \| \frac{x}{\mu_V(x)} - \frac{y}{\mu_V(y)} \bigg \| & = \frac{\| \left( \mu_V(y)-\mu_V(x) \right) x + \mu_V(x)(x-y)\|}{\mu_V(x) \mu_V(y)} \leq \frac{|\mu_V(x)-\mu_V(y)| \| x \|}{\mu_V(x) \mu_V(y)} +\frac{\| x-y\| }{\mu_V(y)} \\
& \leq \frac{r^{-1}\| x-y\| \|x\|}{\mu_V(x)\mu_V(y)} + \frac{\| x-y\| }{\mu_V(y)} \leq \frac{1}{\mu_V(y)}\left( 1 +R r^{-1} \right) \| x-y\|. \nonumber 
\end{align}
Given $x,y\in X \setminus \lbrace 0 \rbrace,$ let us denote $\overline{x}= \frac{x}{\mu_V(x)}$ and $\overline{y}= \frac{y}{\mu_V(y)}.$ Using first \eqref{comparisongradientminkowskinormal}, then \eqref{geometricestimationc11normal} and finally \eqref{previousestimationinverseminkowski} we get
\begin{align*}
\|  \nabla \mu_V(x) & - \nabla \mu_V(y)\|  = \bigg \| \frac{N_S(\overline{x})}{\langle N_S(\overline{x}), \overline{x} \rangle} -  \frac{N_S(\overline{y})}{\langle N_S(\overline{y}), \overline{y} \rangle} \bigg \| \\
& = \frac{\| \left(\langle N_S(\overline{y}), \overline{y} \rangle-\langle N_S(\overline{x}), \overline{x} \rangle \right)N_S(\overline{x}) + \langle N_S(\overline{x}), \overline{x} \rangle\left( N_S(\overline{x})-N_S(\overline{y}) \right) \|  }{\langle N_S(\overline{x}), \overline{x} \rangle \langle N_S(\overline{y}), \overline{y} \rangle} \\
& \leq \frac{| \langle N_S(\overline{y})-N_S(\overline{x}), \overline{x} \rangle| + | \langle N_S(\overline{y}), \overline{y}-\overline{x} \rangle|}{\langle N_S(\overline{x}), \overline{x} \rangle \langle N_S(\overline{y}), \overline{y} \rangle} + \frac{\lip(N_S)}{\langle N_S(\overline{y}), \overline{y} \rangle} \| \overline{x}-\overline{y}\| \\
& \leq \frac{\left( 1 + \| \overline{x}\| \lip(N_S) \right) \| \overline{x}-\overline{y}\|}{\langle N_S(\overline{x}), \overline{x} \rangle \langle N_S(\overline{y}), \overline{y} \rangle} + \frac{\lip(N_S)}{\langle N_S(\overline{y}), \overline{y} \rangle} \| \overline{x}-\overline{y}\| \leq \left( \frac{1 + R \lip(N_S) }{r^2} + \frac{\lip(N_S)}{r} \right)\| \overline{x}-\overline{y}\| \\
& \leq   \frac{\left( \frac{1 + R \lip(N_S) }{r^2} + \frac{\lip(N_S)}{r} \right)}{\mu_V(y)}\left( 1 +R r^{-1} \right) \| x-y\|.
\end{align*}
This proves that, for every $\alpha >0,$ there exists a constant $M_\alpha>0$ such that
$$
\lip( \nabla \mu_V, U_\alpha) \leq M_\alpha, \quad \text{where} \quad U_\alpha= \lbrace z\in X \: : \: \mu_V(z) \geq \alpha \rbrace,
$$
which shows $(5).$ 

\medskip

\noindent $(5) \implies (1):$ By assumption we have that $\nabla \mu_V$ is Lipschitz on $S.$ Since $N_S= \nabla \mu_V / \| \nabla \mu_V\| $ on $S$ we can write, for every $x,y\in S,$ 
$$
 \| N_S(x)-N_S(y)\| \leq \frac{2\| \nabla \mu_V(x)- \nabla \mu_V (y)\|}{\| \nabla \mu_V(y)\|} \leq \frac{2 \lip( \nabla \mu_V, S)}{\|\nabla \mu_V(y)\|} \| x-y\| \leq 2 R \lip( \nabla \mu_V, S)\| x-y\|,
$$
where the last inequality follows from \eqref{infimumminkowskigradient}. We have thus shown that $N_S$ is Lipschitz on $S.$

\subsection{Proof of Theorem \ref{regularity of the signed distance} $(2), \:(3)$ and $(4)$}  We start with the proof of $(2).$ Let $\varepsilon \in (0,1)$ and let $x,y \in \interior(V)$ be such that $d_S(x), d_S(y) \leq (1-\varepsilon)\lip(N_S)^{-1}.$ By Theorem \ref{characterization of C11 convex bodies} $(2),$ the point $P_S(y)$ does not belong to the open ball centered at $P_S(x)-\lip(N_S)^{-1}N_S(P_S(x))$ and with radius $\lip(N_S)^{-1}.$ This is equivalent to
$$
\| P_S(x)-P_S(y) \|^2 \geq 2 \lip(N_S)^{-1} \langle P_S(x)-P_S(y), N_S(P_S(x)) \rangle.
$$
We learnt from $(1)$ that $P_S(x)-x = d_S(x) N_S(P_S(x))$. Using that $d_S(x) \leq (1-\varepsilon)\lip(N_S)^{-1},$ the above inequality yields
\begin{equation}\label{inequalityXpsxpsy}
(1-\varepsilon) \| P_S(x)-P_S(y) \|^2 \geq 2\langle P_S(x)-P_S(y), P_S(x)-x \rangle.
\end{equation}
Similary we deduce
\begin{equation}\label{inequalityYpsxpsy}
(1-\varepsilon) \| P_S(x)-P_S(y) \|^2 \geq 2\langle P_S(y)-P_S(x), P_S(y)-y \rangle.
\end{equation}
After summing \eqref{inequalityXpsxpsy} and \eqref{inequalityYpsxpsy} and making some elementary calculations we get
\begin{equation}\label{Pfirmlynonexpansiveinterior}
\langle P_S(x)-P_S(y), x- y \rangle \geq \varepsilon \|  P_S(x)-P_S(y) \|^2.
\end{equation}
Now, observe that, if $z\in S$ and $w\in \interior(V)$ with $d_S(w) <\lip(N_S)^{-1}$, then
$$
\|z-P_{S }(w)\|\leq  \|z-w\|+\|w-P_S(w)\|\leq \|z-w\| +\| w-z\|=2\|z-w\|.
$$
This fact together with \eqref{Pfirmlynonexpansiveinterior} tell us that $P_S$ is continuous on the set $b_V^{-1} \left( [-(1-\varepsilon)\lip(N_S)^{-1},0] \right)$ and, consequently, \eqref{Pfirmlynonexpansiveinterior} holds for every $x,y$ belonging to this set. 

Finally, recall that the metric projection onto a convex set in a Hilbert space is firmly non-expansive (see \cite[Proposition 4.8]{BauschkeCombettesbook} for instance), which implies that
$$
\langle P_S(x)-P_S(y), x- y \rangle \geq  \|  P_S(x)-P_S(y) \|^2 \quad \text{for all} \quad x,y\in X \setminus \interior(V).
$$
All these observations allow us to conclude $\langle P_S(x)-P_S(y), x- y \rangle \geq \varepsilon \|  P_S(x)-P_S(y) \|^2$ for every $x,y \in U_\varepsilon.$

\medskip

The following Claim will be helpful in the proof of $(3)$.

\begin{claim}\label{claimprojectionanddistanceofinteriorpoints}
{\em Let $r< \lip(N_S)^{-1}$ and $z \in S.$ Then for $0 \leq t\leq r,$ we have that $P_S(z-tN_S(z)))=z$ and $b_V(z-tN_S(z))=-t.$ }
\end{claim}
\begin{proof}
If $0 \leq t \leq r,$ the distance from $z-tN_S(z)$ to $S$ is attained at a unique point by $(1).$ On the other hand, $B(z-rN(z), r) \cap S= \lbrace z \rbrace$ by Theorem \ref{characterization of C11 convex bodies} $(2)$ and, if $y\in S,$ we have
$$
\| y-(z-tN_S(z)) \| \geq \| y-(z-rN_S(z))\| - |r-t| \geq r-(r-t)=t,
$$
with identity if and only if $y=z.$ This shows that $b_V( z-tN_S(z))=-\dist(z-tN_S(z), S)=-t$ and $z=P_S(z-tN_S(z)).$ 
\end{proof}

Let us now proceed with the proof of $(3).$ 

\noindent $(3):$ If $x\in X \setminus V,$ the convexity of $V$ implies that $b_V$ is differentiable at $x$ with $\nabla b_V(x)= \tfrac{x-P_S(x)}{b_V(x)},$ and using $(1)$ we obtain the formula $\nabla b_V = N_S \circ P_S$ on $X \setminus V.$ Now assume that $x\in  V$ is such that $b_V(x)>-\lip(N_S)^{-1}$ and let us prove the differentiability of $b_V$ at $x$ with $\nabla b_V(x)=N_S(P_S(x)).$ Observe that $b_V$ is $1$-Lipschitz on $X$ and the norm $\| \cdot \|$ on $X$ is (Fr\'{e}chet) differentiable at $N_S(x)$ with $\| N_S(x)\|=1$ and $\nabla (\| \cdot \|) ( N_S(x))=N_S(x).$ We can use a theorem of Fitzpatrick's \cite[Theorem 2.4]{Fitzpatrick} which tells us that the $1$-Lipschitz function $b_V$ will be differentiable at $x$ with $\nabla b_V(x)= N_S(x)$ as soon as we check that
\begin{equation}\label{directionaldifferentiabilitysigneddistance}
\lim_{t \to 0} \frac{b_V(x+tN_S(P_S(x)))-b_V(x)}{t}=1.
\end{equation}
Assume first that $x\in \partial V= S.$ If $t>0,$ then $x+tN_S(x) \in X \setminus V$ and $P_S( x+ tN_S(x))=x,$ which shows that $b_V(x+tN_S(x))=t.$ Hence $b_V(x+tN_S(x))-b_V(x)=t$ and \eqref{directionaldifferentiabilitysigneddistance} holds when $t\to 0^+.$ On the other hand, if $r>0$ is such that $r<\lip(N_S)^{-1}$ and $t \in [-r,0),$ we can apply the last part of Claim \ref{claimprojectionanddistanceofinteriorpoints} to obtain that $b_V(x+tN_S(x))=t.$ Thus \eqref{directionaldifferentiabilitysigneddistance} trivially holds when $t \to 0^{-}.$

Let us now check \eqref{directionaldifferentiabilitysigneddistance} for points $x\in \interior(V)$ with $\dist(x,S)< \lip(N_S)^{-1}.$ Take $0<\varepsilon < \dist(x,S)$ such that $ \dist(x,S)+\varepsilon < \lip(N_S)^{-1},$ define $r:= \dist(x,S) + \varepsilon$, and let $0<|t| \leq \varepsilon.$ We have $x-P_S(x)=-\dist(x,S) N_S(P_S(x))$ by virtue of $(1)$ and 
$$
x +t N_S(P_S(x)) = P_S(x)-( \dist(x,S)-t) N_S(P_S(x)),
$$
where $\dist(x,S)-t \in [0,r)$ thanks the choice of $r$ and $\varepsilon.$ Applying the last part of Claim \ref{claimprojectionanddistanceofinteriorpoints}, we obtain 
$$
b_V( x +t N_S(P_S(x)) = b_V \left( P_S(x)-( \dist(x,S)-t) N_S(P_S(x)) \right) = t-\dist(x,S).
$$
This immediately yields \eqref{directionaldifferentiabilitysigneddistance}. In conclusion, we have shown that $b_V$ is Fr\'echet differentiable at every $x\in X$ such that $b_V(x)>- \lip(N_S)^{-1},$ with $\nabla b_V(x) = N_S  P_{S}(x).$ Moreover, the mapping $U_\varepsilon \ni x \mapsto P_S(x)$ is $ \varepsilon^{-1}$-Lipschitz by $(2)$, and therefore
$$
\|\nabla b_V(x)- \nabla b_V (y) \| \leq \lip(N_S) \lip(P_{S}) \| x-y\| \leq \tfrac{1}{\varepsilon} \lip(N_S)\|x-y\|
$$ 
for every $x,y\in U_\varepsilon.$

\medskip

\noindent $(4):$ Outside $V$ we have that $b_V=\dist( \cdot, V)$, and $\dist( \cdot, V)$ is convex on $X.$ Hence $b_V$ is convex on any line segment contained in $X\setminus \interior{V}$. Let us now see that $b_V$ is convex on $\interior(V)$. If $[x,y]$ is a line segment contained in $\interior(V)$ and
$$
z_{\lambda}:=(1-\lambda)x+\lambda y, \quad \lambda \in [0,1],
$$
is a point of $[x,y],$ for every $\varepsilon>0$ we can find a point $p_\lambda \in S$ such that
$$
\|z_\lambda-p_\lambda\| \leq \dist(z_\lambda, S) + \varepsilon = -b_V(z_\lambda) +\varepsilon .
$$
Let $W_\lambda$ denote the tangent hyperplane to $S$ at $p_\lambda$; since $V$ is convex we have that $W_\lambda \cap\interior{V}=\emptyset$. Then, if $p_x$ and $p_y$ denote the orthogonal projections of $x$ and $y$ onto $W_\lambda$, we have $p_x, p_y \in X\setminus V$, and therefore
$$
\dist(x, S)\leq \|x-p_x\| \quad \textrm{and} \quad \dist(y, S)\leq \|y-p_y\|.
$$
On the other hand, the function
$$
[0,1]\ni t\mapsto  \dist \left( (1-t)x+ty, W_\lambda \right)
$$
is obviously affine, so we have
\begin{align*}
 -& b_V(z_\lambda) + \varepsilon  \geq \|z_\lambda-p_\lambda\| \geq \dist(z_\lambda, W_\lambda) =(1-\lambda)\dist(x, W_\lambda)+\lambda \dist(y, W_\lambda) \\
 & = (1-\lambda)\|x-p_x\|+\lambda \|y-p_y\|  \geq (1-\lambda) \dist(x, S)+\lambda \dist(y, S)  = -(1-\lambda)b_{V}(x)-\lambda b_{V}(y),
\end{align*}
that is to say,
$$
b_{V}\left( (1-\lambda)x+\lambda y\right) = b_V(z_\lambda) \leq (1-\lambda)b_{V}(x)+\lambda b_{V}(y) + \varepsilon.
$$
Letting $\varepsilon \to 0^+,$ the above argument shows that $b_{V}$ is convex on $\interior{V}$, and by continuity it follows that $b_{V}$ is convex on $V$. Finally, if $x\in X\setminus V$ and $y\in \interior{V}$, hence the line segment $[x,y]$ is transversal to $S$, we may write $[x,y]=[x,z]\cup [z,y]$, where $z\in S,$ $[x,z]\subset X\setminus\interior{V} $ and $[z, y]\subset V$. Consider the function $\varphi:[0,1]\to\R$ defined by
$\varphi(t)=b_{V}\left( (1-t)x + ty\right)$, and let $t_0\in (0,1)$ be the number such that $z=(1-t_0)x+t_0 y$. We know that $\varphi$ is convex on $[0, t_0]$, and $\varphi$ is convex on $[t_0, 1]$ as well. Besides $\varphi$ is differentiable at $t_0$ because $b_V$ is differentiable on a neighbourhood of $S$ by $(3)$. Hence $\varphi$ is convex on $[0,1]$, for every $x,y$. It follows that $b_V$ is convex on $[x,y]$. Therefore $b_V$ is convex on $X$.


\section{Main results}\label{sectionmainresults}

In this section we will establish some generalizations of Theorem \ref{main theorem for C11 Hilbert} which are valid for convex bodies of class $C^{1, \omega}$ in Hilbert spaces or for convex bodies of class $C^{1, \alpha}$ in Banach spaces with equivalent norms of power type $1+\alpha$, with $\alpha\in (0,1]$. Of course, the usual norm of any Hilbert space satisfies this property with $\alpha=1$. In fact, it is well known that superreflexive Banach spaces are characterized as being Banach spaces with equivalent norms of class $C^{1, \alpha}$ for some $\alpha\in (0, 1]$, and Hilbert spaces are characterized as being Banach spaces with equivalent norms of class $C^{1,1}$. For reference about renorming properties of superreflexive spaces see \cite{Pisier, DGZ, FabianEtAl}.

But we must first specify what we mean by a convex body of class $C^{1, \alpha}$, $0<\alpha\leq 1$, in a Banach space, or more generally, by a convex body of class $C^{1, \omega}$, where $\omega$ is a modulus of continuity. 

The first difficulty we encounter is that Definition \ref{definition of C11 hypersurface} no longer makes sense in a non-Hilbertian Banach space, as we do not have a notion of orthogonality in this setting. For the same reason, the statement of Theorem \ref{characterization of C11 convex bodies} does not make sense in a Banach space.

On the other hand, even if we should like to restrict our investigation to Hilbert spaces $X$, it is unclear what convex bodies in $X$ should be called of class  $C^{1, \omega}$ (where $\omega$ is a modulus of continuity). As a matter of fact, there are no analogues of Theorems \ref{characterization of C11 convex bodies} and \ref{regularity of the signed distance} for the class $C^{1, \alpha}$ when $\alpha<1$. This can be shown by considering a bounded convex body $W$ in $\R^2$ such that $0\in \textrm{int}(W)$ as an interior point and such that the graph of $y=|x|^{3/2}-1$, $-2\leq x\leq 2$, is contained in $\partial W$, and $\partial W$ is $C^{\infty}$ smooth away from the point $(0,-1)$. The Minkowski functional $\mu_W$ of such a body will be of class $C^{1, 1/2}$ on the set $\{(x,y) : 1/2 <\mu_{W}(x,y)<2\}$ (see the proof of $(1) \implies (5)$ in Theorem \ref{characterization of C11 convex bodies}), and the outer normal $N_{\partial W}$ will be $1/2$-H\"older continuous, so we are tempted to call $W$ a $C^{1, 1/2}$ convex body; however, property $(2)$ of Theorem \ref{characterization of C11 convex bodies}, as well as properties $(1)$ and $(3)$ of Theorem \ref{regularity of the signed distance}, will fail for this body $W$. Since $W$ is bounded and $\mu_W$ is $C^{1, 1/2}$ it is easy to see that $W$ still satisfies $(3)$ of Theorem \ref{characterization of C11 convex bodies} for a $C^{1,1/2}$ convex function $\varphi$.

In view of these facts, at least from an analytical point of view, and with the purpose of solving Problem \ref{main problem} for the classes $\mathcal{C}=C^{1, \omega}$ in a Hilbert space, or $\mathcal{C}=C^{1, \alpha}$ in a superreflexive space, we consider that, among all the available options, the following definition is the most satisfactory.

\begin{definition}\label{definition of C1omega convex body}
{\em Let $S$ be a subset of a Banach space $X$. We will say that $S$ is a convex hypersurface of class $C^{1, \alpha}$, where $\alpha\in (0, 1]$, provided that  there exist a number $M>0$ and a convex function $F \in C^{1, \alpha}(X)$ such that $S=F^{-1}(1)$ and 
$$
 M^{-1} \leq \|D F(x)\|_*\leq M \: \textrm{ whenever } \: x\in S.
$$
More generally, if $\omega$ is a modulus of continuity, we will say that a subset $S$ of $X$ is a convex hypersurface of class $C^{1, \omega}$ if there exist a number $M>0$ and a convex function $F\in C^{1, \omega}(X)$ such that $S=F^{-1}(1)$ and
$$
M^{-1}\leq \|D F(x)\|_* \leq M \:  \textrm{ whenever } \: x\in S.
$$
}
\end{definition}
By a modulus of continuity $\omega$ we will understand a concave and strictly increasing function $\omega : [0, + \infty) \to [0, + \infty)$ such that $\omega(0)=0$ and $\lim_{t \to +\infty} \omega(t)=+\infty$. Observe that such a function $\omega$ has a well defined inverse $\omega^{-1}:[0, \infty)\to [0, \infty)$ which is convex, strictly increasing, and satisfies $\omega^{-1}(0)=0$ and $\lim_{s \to +\infty} \omega^{-1}(s)=+\infty$.

\medskip

The following result generalizes Theorem  \ref{characterization of C11 convex bodies} to a large extent, but it does not provide sharp constants in $(4)$.

\begin{theorem}\label{characterizationC1omegaconvexbody}
Let $S$ be a convex hypersurface of a Hilbert space $X$, say $S=\partial V$, where $V$ is a closed convex body (not necessarily bounded). Assume that $S$ is a $C^1$ submanifold, so that the outer unit normal $N_S:S\to S_X$ is well defined. Let $\omega$ be a modulus of continuity, and denote $\varphi(t):=\int_0^t \omega (s) ds$. Then, the following statements are equivalent:
\begin{enumerate}
\item There exists $M>0$ such that $\| N_S(x)-N_S(y)\| \leq M \omega(\|x-y\|)$ for every $x,y\in S.$
\item There exists $M>0$ such that $W_x:= \lbrace p\in X \: \: \: \langle N_S(x),x-p\rangle \geq M \varphi(2\|x-p\|) \rbrace \subseteq V$ with $S \cap W_x = \lbrace x \rbrace$ for every $x\in S$. 
\item There exists $M>0$ such that $\langle N_S(y), y-x \rangle \geq \tfrac{\|N_S(x)-N_S(y)\|}{2} \omega^{-1}\left( \tfrac{\| N_S(x)-N_S(y)\|}{4M} \right)$ for every $x,y\in S.$ 
\item There exists a {\em convex} function $F:X\to\R$ of class $C^{1,\omega}$ such that $S=F^{-1}(1)$ and $\nabla F(x)=N_S(x)$ for every $x\in S$.
\item $S$ is a convex hypersurface of class $C^{1, \omega}.$
\end{enumerate}
Furthermore, if $V$ is bounded and $0\in \interior(V)$, then the above statements are also equivalent to:
\begin{enumerate}
\item[{(6)}] For every $\alpha >0$, $\mu_V$ is of class $C^{1,\omega}$ on the set $\{x\in X: \mu_V(x) \geq \alpha\}$ 
\end{enumerate}
\end{theorem}
\begin{proof}
\noindent $(1) \implies (2):$ Let $x,y$ be two different points in $S$, and assume that $y\in W_x.$ Then we have
$$
0 \leq \langle N_S(y), y-x \rangle = \langle N_S(y)-N_S(x),y-x \rangle + \langle N_S(x), y-x \rangle \leq M \omega(\|x-y\|)\|x-y\| -M\varphi(2\|x-y\|),
$$
where the last term is negative since $\varphi(2t) > t \omega(t)$ for every $t>0.$ This proves that $W_x \cap S= \lbrace x \rbrace.$ Now, observe that $x-\varepsilon N_S(x)$ belongs to $\interior(V) \cap \interior(W_x)$ for $\varepsilon>0$ small enough. Thus $W_x$ and $V$ are two convex bodies such that $W_x \cap \partial V$ is a single point and $\interior(V) \cap \interior(W_x) \neq \emptyset.$ Therefore $W_x \subset V.$ 

\medskip

\noindent $(2) \implies (3):$ Let $x,y\in S$ and define $r:=\| N_S(y)-N_S(x)\|$. We may assume that $r>0$, as $(3)$ trivially holds when $N_S(x)=N_S(y)$. Also set $p:=x+ \tfrac{1}{r} \omega^{-1}\left( \tfrac{r}{4M} \right)(N_S(y)-N_S(x)).$  Bearing in mind that $2t \omega(t) \geq \varphi(2t)$ for every $t\geq 0$ (which follows from the concavity of $\omega$), we can write
\begin{align*}
\langle N_S(x), x-p \rangle & = \tfrac{1}{r} \omega^{-1}\left( \tfrac{r}{4M} \right) \langle N_S(x), N_S(x)-N_S(y) \rangle = \tfrac{r}{2} \omega^{-1}\left( \tfrac{r}{4M} \right)   \\
& =2M \tfrac{r}{4M} \omega^{-1}\left( \tfrac{r}{4M} \right) \geq M \varphi \left( 2 \omega^{-1}\left( \tfrac{r}{4M} \right) \right)=M \varphi ( 2\|x-p\| ).
\end{align*}
This shows that $p\in W_x,$ which implies that $p\in V$ by virtue of $(2).$ We thus have $\langle N_S(y), y-p \rangle \geq 0$ by convexity of $V.$ Finally, we can write 
\begin{align*}
\langle N_S(y) &,   y-x \rangle = \langle N_S(y),y-p \rangle + \langle N_S(y), p-x \rangle \geq \langle N_S(y), p-x \rangle \\
&  = \tfrac{1}{r} \omega^{-1}\left( \tfrac{r}{4M} \right) \langle N_S(y), N_S(y)-N_S(x) \rangle = \tfrac{r}{2} \omega^{-1}\left( \tfrac{r}{4M} \right)=\tfrac{\|N_S(x)-N_S(y)\|}{2} \omega^{-1}\left( \tfrac{\| N_S(x)-N_S(y)\|}{4M} \right).
\end{align*}

\medskip

\noindent $(3) \implies (4):$ We define $(f,G):= (1,N_S)$ on $S.$ By $(3)$ the jet $(f,G)$ satisfies the inequality 
$$
f(x) \geq f(y) + \langle G(y), x-y \rangle + \tfrac{\|G(x)-G(y)\|}{2} \omega^{-1}\left( \tfrac{\| G(x)-G(y)\|}{4M} \right), \quad x,y \in S,
$$
and then \cite[Theorem 4.11]{AzagraMudarraExplicitFormulas} provides us with a convex function $F:X \to \R$ of class $C^{1,\omega}$ with $F=1$ and $\nabla F= N_S$ on $S.$ The same argument as in the proof of Theorem \ref{characterization of C11 convex bodies} gives that, in fact, $F^{-1}(1)=S.$ 

\medskip

\noindent $(4) \implies (5):$ This is obvious from Definition \ref{definition of C1omega convex body}.

\medskip

\noindent $(5) \implies (1).$ Let $F$ be a function as in Definition \ref{definition of C1omega convex body}. Of course we have $N_S= \nabla F / \| \nabla F \|$ on $S$, and then
$$
\| N_S(x)- N_S(y) \| \leq 2 \frac{\| \nabla F(x)-\nabla F(y)\|}{\|\nabla F(y)\|} \leq \frac{2}{\inf_{S} \| \nabla F \|} \lip(\nabla F) \omega \left( \| x-y\| \right) 
$$
for every $x,y\in S.$ This shows $(1).$ 

\medskip

The proofs of $(1) \implies (6)$ and $(6) \implies (1)$ in the case that $V$ is bounded are similar to those of Theorem \ref{characterization of C11 convex bodies}.
\end{proof}

\medskip

The next two results generalize Theorem \ref{main theorem for C11 Hilbert}.

\begin{theorem}\label{main theorem for C1omega Hilbert}
Let $C$ be a subset of a Hilbert space $X$, and let $N:C\to S_X$ be a mapping. Then the following statements are equivalent.
\begin{enumerate}
\item There exists a $C^{1, \omega}$ convex body $V$ such that $C\subseteq \partial V$ and $N(x)$ is outwardly normal to $\partial V$ at $x$ for every $x\in C$.
\item There exists some $\delta>0$ such that
$$
\langle N(y), y-x\rangle \geq  \|N(x)-N(y)\|  \omega^{-1}\left(  \delta \|N(x)-N(y)\| \right)  \quad \textrm{for all} \quad x, y\in C.
$$
\end{enumerate}
Moreover, if we further assume that $C$ is bounded, then $V$ can be taken to be bounded as well.
\end{theorem}

\begin{theorem}\label{main theorem for C1alpha superreflexive}
Let $C$ be a subset of a superreflexive Banach space $X$ such that $X$ has an equivalent differentiable norm with modulus of smoothness of power type $1+\alpha$, where $\alpha\in (0,1]$. Let us denote by $X^{*}$ (resp. by $S^{*}$) the dual space of $X$, endowed with the dual norm $\|\cdot\|_{*}$ of $\|\cdot\|$ (resp. the dual sphere of $(X, \|\cdot\|)$). Let $D:C\to S^{*}$ be a mapping. Then the following statements are equivalent.
\begin{enumerate}
\item There exists a $C^{1, \alpha}$ convex body $V$ such that $C\subseteq \partial V$ the hyperplane $H_x:=\{y\in X : D(x)(y)=D(x)(x)\}$ is tangent to $\partial V$ at $x$ and $V \subseteq H_x^{-}:=\{y\in X : D(x)(y)\leq D(x)(x)\}$ for every $x\in C$.
\item There exists some $\delta>0$ such that
$$
D(y)(y-x) \geq \delta \|D(x)-D(y)\|_{*}^{1+\frac{1}{\alpha} } \quad \textrm{for all} \quad x, y\in C.
$$
\end{enumerate}
Moreover, if we further assume that $C$ is bounded, then $V$ can be taken to be bounded as well.
\end{theorem}

Finally, let us observe that the above theorems cannot be extended to Banach spaces which are not superreflexive.

\begin{remark}
{\em Assume that Theorem \ref{main theorem for C1omega Hilbert} is true for a Banach space $(X, \|\cdot\|)$. Pick a point $x_{0}\in X\setminus\{0\}$, and a linear form $\xi_0\in S^*$. Then condition $(2)$ of Theorem \ref{main theorem for C1omega Hilbert} is trivially satisfied for $C:=\{x_0\}$ and $D(x_0) :=\xi_0$. Therefore there exists a {\em bounded} $C^{1, \omega}$ convex body $W$ such that $W$ is of class $C^{1, \omega}$. Up to a translation we may assume that $0\in\interior{W}$. Hence the Minkowski functional of $W$, denoted by $\mu_W$ is subadditive, positively homogeneous, and satisfies $\mu_W(x)=0 \iff x=0$. Moreover, with the same proof as in $(1) \implies (5)$ of Theorem \ref{characterization of C11 convex bodies} we obtain that $\mu_W$ is of class $C^{1,\omega}$ on the superlevel sets $\lbrace x\in X \: : \: \mu_W(x) \geq \alpha \rbrace$ for every $\alpha>0.$ If $0<r\leq R$ are such that $B(0,r) \subset W \subset B(0,R),$ then $R^{-1} \|\cdot \| \leq \mu_W \leq r^{-1} \| \cdot \|$ on $X$ and hence the function
$$
\rho(x) :=\mu_W (x)+ \mu_W (-x)
$$
defines an equivalent norm in $X$ which is uniformly differentiable on its unit sphere, and this implies that $X$ is superreflexive; see \cite{DGZ} for instance.
}
\end{remark}

\section{Proofs of the main results}

In this section we will prove Theorems \ref{main theorem for C11 Hilbert}, \ref{main theorem for C1omega Hilbert} and \ref{main theorem for C1alpha superreflexive}. 

\subsection{Proof of Theorem \ref{main theorem for C11 Hilbert}}

If $V$ is a $C^{1,1}$ convex body whose outer unit normal is $L$-Lipschitz, we know from Theorem \ref{characterization of C11 convex bodies} (3) that the inequality of $(2)$ in Theorem \ref{main theorem for C11 Hilbert} is satisfied with $L=r^{-1},$ for every $x,y\in \partial V.$ 

\medskip

Conversely, let us assume that $(2)$ is satisfied for $C \subset X, \: N:C \to S_X$ and $r>0.$ For every $y\in C,$ we define $B_y:= B(y-rN(y),r).$ Let us define
$$
V:= \overline{\co}\left( \bigcup_{y\in C} B_y \right),
$$
that is the closed convex hull of the union of the balls $B_y$. Obviously, we have $C\subset V$. Let us first see that in fact $C\subset\partial V$. Suppose that $y\in C \cap \interior(V).$ Then $y$ can be written as $y= \sum_{i=1}^n\lambda_i w_i;$ where $w_i \in \interior(B_{y_i}), \: y_i\in C, \: \lambda_i \geq 0,$ for every $i=1,\ldots,n, \: \sum_{i=1}^n \lambda_i=1$ and $n\in \N.$ By the assumption we have
$$
\langle N(y), y-y_i \rangle \geq \tfrac{r}{2} \| N(y)-N(y_i)\|^2 ,  \quad i=1,\ldots,n.
$$
This is equivalent to
$$
\langle N(y), y-z_i \rangle \geq r, \quad \text{where} \quad z_i:=y_i-rN(y_i), \quad i=1,\ldots,n.
$$
We obtain that $\langle N(y), y-\sum_{i=1}^n \lambda_i z_i \rangle \geq r,$ where $\| y - \sum_{i=1}^n \lambda_i z_i \| \leq \sum_{i=1}^n \lambda_i \| w_i-z_i\|<r,$ a contradiction. Hence $C \subseteq \partial V.$ Now we claim the following.

\begin{claim}\label{claimballrollsfreelysuficiency}
{\em For every $x\in \partial V$ there exists $z_x\in V$ such that $B(z_x,r) \subset V$ and $x\in \partial B(z_x,r).$ }
\end{claim}
\begin{proof}
If $y\in \co \left( \bigcup_{x\in C} B_x \right)$, then $y$ can be written as $y= \sum_{i=1}^n \lambda_i w_i,$ where $\lambda_i \geq 0$ and $w_i \in B_{y_i}$ for every $i=1, \ldots,n,$ $\sum_{i=1}^n \lambda_i=1$ and $n\in \N.$ Set $z_i:=y_i-rN(y_i)$ (the center of $B_{y_i}$), for $i=1, \ldots,n$, and $z:= \sum_{i=1}^n \lambda_i z_i$. Given any point $p \in B(z,r),$ it is clear that $p= \sum_{i=1}^n \lambda_i p_i,$ where $p_i:= p-z+z_i$ and $\|p_i-z_i\| = \| p-z\|\leq r$, hence $p_i\in B_{y_i}$ for every $i=1, \ldots, n$. This shows that $p\in \co \left( \bigcup_{x\in C} B_x \right) \subset V$ and therefore $B(z,r) \subset V$ with $\| z-y\| \leq \sum_{i=1}^n \lambda_i \| z_i-w_i\| \leq r.$ 

Now, let $x\in \partial V$ and consider a sequence $(y_k)_k \subset \co \left( \bigcup_{y\in C} B_y \right)$ converging to $x.$ By the above argument we can find a sequence $(z_k)_k$ on $V$ such that $y_k \in B(z_k,r) \subset V$ for every $k.$ Up to passing to a subsequence, we may assume that $(z_k)_k$ weakly converges to some $z_x \in V.$ Let us see that $B(z_x,r) \subset V.$ Indeed, otherwise there exist $u \in X \setminus \lbrace 0 \rbrace, \: \alpha \in \R$ and $w \in B(z_x,r)$ such that $V \subset \lbrace \langle u, \cdot \rangle \leq \alpha \rbrace$ and $\langle u, w \rangle > \alpha.$ The point $w_k:=w+z_k-z_x$ belongs to $B(z_k,r) \subset V$ for every $k$ and $(w_k)_k$ weakly converges to $w.$ Thus we have $\alpha \geq \lim_k \langle u, w_k \rangle = \langle u, w \rangle > \alpha,$ a contradiction. Therefore $B(z_x,r) \subset V.$ Also, observe that $(z_k-y_k)_k$ weakly converges to $(z_x-x),$ where $\| z_k-y_k\| \leq r$ for every $k$, and because $B(0,r)$ is weakly closed, we have $\|z_x-x\|\leq r$, that is, $x\in B(z_x,r).$ In fact, $x\in \partial B(z_x,r)$ because $x\in \partial V$ and $B(z_x,r) \subset V.$
\end{proof}

Now, let us see that $\partial V$ is a $C^1$ manifold. We can assume without loss of generality that $0 \in \interior(V).$ For every $x\in \partial V,$ take $z_x$ as in Claim \ref{claimballrollsfreelysuficiency} and define $g(y):=r^{-1}\|y-z_x\|$ for every $y\in X.$ Observe that $\mu_V(y)=\mu_{V-z_x}(y-z_x)$ and $g(y)=\mu_{B(z_x,r)-z_x}(y-z_x)$ for every $y\in X.$ Then, because $B(z_x,r) \subset V$ and $x\in \partial V \cap \partial B(z_x,r),$ $\mu_V$ and $g$ are two continuous convex functions such that $\mu_V \leq g$ on $X$ and $\mu_V (x)=g(x)=1.$ Since $g$ is differentiable at $x$, we conclude that $\mu_V$ is differentiable at $x$ too with $\nabla \mu_V(x)=\nabla g(x)= r^{-1}(x-z_x)/\|x-z_x\|.$ We have shown that $\mu_V$ is differentiable on $\partial V,$ and by homogeneity, $\mu_V$ is differentiable on an open neighbourhood of $\partial V.$ In conclusion $V$ is a $C^1$ manifold.

To see that $\partial V$ is of class $C^{1,1}$ with $\lip(N_{\partial V}) \leq r^{-1},$ it is now enough to apply Theorem \ref{characterization of C11 convex bodies} $(2)$ in combination with Claim \ref{claimballrollsfreelysuficiency}.

Finally, if $y\in C,$ observe that, by definition of $V,$ the point $z_y:=y-rN(y)$ is such that Claim \ref{claimballrollsfreelysuficiency} is true for the ball $B(z_y,r).$ Using the above argument we obtain (assuming that $0\in \interior(V)$) that $\nabla \mu_V(y)=r^{-1}(y-z_y)/\|y-z_y\| = r^{-1}N(y).$ In consequence $N$ coincides with the outer unit normal $N_{\partial V}$ to $\partial V$ at points of $C.$ This completes the proof of Theorem \ref{main theorem for C11 Hilbert}.

\subsection{Proof of Theorem \ref{main theorem for C1omega Hilbert}} 

It is clear that $(1$) implies $(2)$ from the characterizations provided in Theorem \ref{characterizationC1omegaconvexbody}.

\medskip

Conversely, let us assume that $(2)$ is satisfied. Let us define $\varphi(t) := \int_0^t \omega(s) ds$ for every $t\geq 0.$ The Fenchel conjugate of $\varphi$ is defined by $$\varphi^*(t)=\int_0^t \omega^{-1}(s)ds $$ for every $t\geq 0,$ and it is clear that $\varphi^*(t)\leq t \omega^{-1}(t).$ By assumption we have
$$
\langle N(y),y-x \rangle \geq \| N(x)-N(y)\| \omega^{-1}\left( \delta \| N(x)-N(y)\| \right) \geq  \delta^{-1} \varphi^*\left( \delta \| N(x)-N(y)\| \right), \quad x,y\in C.
$$
Therefore, the jet $(f,G):=(1,N)$ satisfies the inequality
$$
f (x) \geq f(y) + \langle G (y), x-y \rangle + \delta^{-1} \varphi^*\left( \delta \| G(x)-G(y)\| \right) \quad \text{for every} \quad x,y\in C.
$$
According to \cite[Theorem 4.11]{AzagraMudarraExplicitFormulas}, the function
$$
H:=\textrm{conv}(g), \quad \text{where} \quad g(x)=\inf_{y\in C} \lbrace 1+ \langle N(y),x-y \rangle + \delta^{-1}\varphi(\|x-y\|) \rbrace, \quad x\in X,
$$
is convex and of class $C^{1,\omega}(X)$ with $H=1$ and $\nabla H=N$ on $C.$ Bearing in mind the identities $\varphi(\omega^{-1}(\delta)) + \varphi^*(\delta) = \delta \omega^{-1}(\delta)$ and $\varphi'=\omega,$ it is easy to see that, for every $y\in C,$ the function $z \mapsto 1+ \langle N(y),z-y \rangle + \delta^{-1}\varphi(\|z-y\|)$ attains its global minimum at $z_y= y-\omega^{-1}N(y)$ and this minimum value is $1-\delta^{-1} \varphi^*(\delta).$ This easily implies
\begin{equation}\label{estimationminimumfunction}
H\left( y- \omega^{-1}(\delta)N(y) \right) = \inf_X H = 1-\delta^{-1} \varphi^*(\delta) \quad \text{for every} \quad y \in C.
\end{equation}
We now define
\begin{equation}\label{definitionsetsmodulifunctions}
A:= \overline{\co}(C \cup \lbrace y-\omega^{-1}(\delta)N(y) \: : \: y\in C \rbrace), \quad  F(x):= H(x)+ \delta^{-1} \varphi^*(\delta) \varphi \left( d_A(x) \right) \quad x\in X,
\end{equation}
where $d_A$ stands for the distance function to $A.$ The function $F$ is convex because so are $H,$ $\varphi$ and $d_A,$ and $\varphi$ is increasing. In addition, we have that
\begin{equation}\label{inequalityHilbertnorm}
\varphi( \| x+h\|) + \varphi( \| x-h\|) - 2 \varphi( \| x\|) \leq \varphi( 2 \| h \| ) , \quad x,h \in X;
\end{equation}
see \cite[Lemma 4.6]{AzagraMudarraExplicitFormulas}. Thus if $x,h\in X$ and $y \in A$ is such that $d_A(x)=\| x-y\|,$ the inequality \eqref{inequalityHilbertnorm} for $x-y$ and $h$ gives
$$
\varphi( d_A(x+h)) + \varphi( d_A(x-h))-2\varphi( d_A(x)) \leq \varphi( \| x+h-y\|) + \varphi(\| x-h-y\|) - 2 \varphi( \| x-y\|) \leq \varphi( 2 \|h\|).
$$
Since $\varphi \circ d_A$ is continuous and convex, the above inequality shows that $\varphi \circ d_A$ is of class $C^{1,\omega}(X);$ see \cite[Proposition 4.5]{AzagraMudarraExplicitFormulas} for a proof of this fact. This shows that $F$ is a $C^{1,\omega}(X)$ convex function. Finally, let us check that $V= F^{-1} (-\infty,1]$ is the desired convex body. By \eqref{estimationminimumfunction}, $V$ is a non-degenerate sublevel set of a differentiable convex function, that is, $V$ is a convex body of class $C^1.$ It is obvious that $C \subseteq  F^{-1}(1) = \partial V$ and the outer unit normal $N_{\partial V}$ to $\partial V$ coincides with $\nabla F / \| \nabla F \| = N$ on $C.$ According to Definition \ref{definition of C1omega convex body} $V$ will be of class $C^{1,\omega}$ as soon as we find $M>0$ such that $M^{-1} \leq \| \nabla F(x) \| \leq M$ whenever $ F(x)=1.$ Given $x\in X$ with $F(x)=1$ and $\varepsilon>0,$ it is easy to see from \eqref{definitionsetsmodulifunctions} that we can find 
\begin{equation}\label{approximationinequality}
z \in \co \lbrace y-\omega^{-1}(\delta)N(y) \: : \: y\in C \rbrace \quad \text{such that} \quad \| x-z \| \leq  d_A(x) +  \omega^{-1}(\delta) + \varepsilon.
\end{equation} 
Since $z\in A,$ we have that $\varphi(d_A(z))=0$ and $\nabla (\varphi \circ d_A) (z)=0.$ Then \eqref{estimationminimumfunction} and the convexity of $H$ give $F(z)=1-\delta^{-1} \varphi^*(\delta)$ and $\nabla F(z)=0.$ Because $H$ is bounded below by $1-\delta^{-1} \varphi^*(\delta),$ it follows from \eqref{definitionsetsmodulifunctions} that $d_A(x) \leq \varphi^{-1}(1).$ Since $\nabla F$ is $\omega$-continuous on $X$, there exists some $L>0$ such that
$$
\| \nabla F(x) \| \leq  \| \nabla F(z) \| + L \omega( \| x-z\| )  = L \omega( \| x-z\| ),
$$
and \eqref{approximationinequality} together with the preceding remarks yield
$$ \| \nabla F(x) \| \leq L \omega \left(  d_A(x) +  \omega^{-1}(\delta) +\varepsilon \right) \leq  L \omega \left(  \varphi^{-1}(1) +  \omega^{-1}(\delta) +\varepsilon \right).
$$
By letting $\varepsilon \to 0^+$ we obtain $\| \nabla F(x) \| \leq L \omega \left(  \varphi^{-1}(1) +  \omega^{-1}(\delta) \right).$ On the other hand, the convexity of $F$ gives
$$
\| \nabla F(x) \| \geq \frac{F(x)-F(z)}{\|x-z\|} = \frac{\delta^{-1} \varphi^*(\delta)}{\|x-z\|} \geq \frac{\delta^{-1} \varphi^*(\delta)}{ \varphi^{-1}(1)+\omega^{-1}(\delta)+\varepsilon},
$$
Letting $\varepsilon \to 0^+,$ we conclude $\| \nabla F(x)\| \geq \delta^{-1} \varphi^*(\delta)\left( \varphi^{-1}(1)+\omega^{-1}(\delta) \right)^{-1}. $ 

In addition, let us see that if $C$ is bounded, the convex body $V$ is also bounded. Indeed, the set $A$ in \eqref{definitionsetsmodulifunctions} is bounded because so is $C$ and because $\|N\|=1$. Thus the function $\varphi \circ d_A$ is coercive, that is, $\lim_{\|x\| \to \infty} \varphi(d_A(x))= +\infty$ (observe that $\varphi$ is coercive because so is $\omega$ and we have the inequality $\varphi(t) \geq \tfrac{t}{2} \omega( \tfrac{t}{2} )$). Since $H$ is bounded below on $X,$ the function $F$ of \eqref{definitionsetsmodulifunctions} is coercive too and therefore $V=F^{-1}(-\infty,1]$ is a bounded subset.

\subsection{Proof of Theorem \ref{main theorem for C1alpha superreflexive}}
Let us first assume that $V$ is the $C^{1,\alpha}$ convex body of $(1)$ and consider a convex function $F \in C^{1,\alpha}(X)$ as in Definition \ref{definition of C1omega convex body}. We know from \cite[Proposition 5.3]{AzagraMudarraExplicitFormulas} that there exists $\delta>0$ such that
\begin{equation}\label{inequalitycw1alphanecessity}
F(x)-F(y)+ DF(y)(y-x) \geq \delta \| DF(x)-DF(y)\|_*^{1+\frac{1}{\alpha}} \quad \text{for every} \quad x,y\in X.
\end{equation}
Also, because $V=F^{-1}(-\infty,1],$ $\partial V =F^{-1}(1)$ and $F$ is $C^1,$ the hyperplane $\lbrace y\in X \: : \: DF(x)(y-x)=0 \rbrace$ is tangent to $\partial V$ at $x$ for every $x\in \partial V.$ By the assumption, $D$ must be a positive multiple of $D F$ and therefore $D(x) = \frac{DF(x)}{\| DF(x)\|_*}$ for every $x\in C.$ We can easily dedude that
 \begin{equation}\label{estimationnormaldifferentialvarphialpha}
 \| D(x)-D(y)\|_* \leq \frac{2 \| D F(x)-D F(y)\|_*}{\| D F(y)\|_*} \quad \text{for every} \quad x,y\in C.
 \end{equation}
By plugging \eqref{estimationnormaldifferentialvarphialpha} in \eqref{inequalitycw1alphanecessity} and bearing in mind that $\inf_{\partial V} \| DF\|_*$ is positive we conclude
$$
D(y)(y-x) \geq \tfrac{\delta}{2^{1+\frac{1}{\alpha}}} \Big ( \inf_{\partial V} \| D F\|_* \Big )^{\frac{1}{\alpha}}\| D(x)-D(y)\|_* ^{1+\frac{1}{\alpha}} \quad \text{for every} \quad x,y\in C.
$$

\medskip

Conversely, assume that $(2)$ is satisfied. Let $L \geq 2$ be a constant such that
\begin{equation}\label{inequalityrenorming}
\| x+h\|^{1+\alpha} + \| x-h\|^{1+\alpha} - 2\| x\|^{1+\alpha} \leq L \| h\|^{1+\alpha}, \quad x,h \in X.
\end{equation}
Since $X$ is reflexive and the norm $\| \cdot\|$ is strictly convex, for every $y\in C$ we can find a unique $N(y) \in S_X$ such that $D(y)(N(y))=1.$ By assumption, the jet $(f,G):=(1,D)$ defined on $C$ satisfies the inequality
$$
f (x) \geq f(y) + G (y)( x-y ) + \delta\| G (y)-G (x)\|_*^{1+\frac{1}{\alpha}}, \quad x,y\in C .
$$
Let $M>0$ be a constant such that $\delta = \tfrac{\alpha}{(1+\alpha) M^{1/\alpha}}.$ According to \cite[Theorem 5.5]{AzagraMudarraExplicitFormulas}, the function
$$
H:=\textrm{conv}(g), \quad \text{where} \quad g(x)=\inf_{y\in C} \lbrace 1+ D(y)(x-y ) + \tfrac{M}{1+\alpha} \|x-y\|^{1+\alpha} \rbrace, \quad x\in X,
$$
is convex and of class $C^{1,\alpha}(X)$ with $H=1$ and $D H=D$ on $C.$ It is easy to see that each function $z \mapsto 1+ D(y)(z-y ) + \tfrac{M}{1+\alpha} \|z-y\|^{1+\alpha}$ attains its global minimum at the point $z_y= y-M^{-1/\alpha} N(y)$ and this minimum value is $1-\tfrac{\alpha}{1+\alpha}M^{-1/\alpha} = 1-\delta.$ This shows that
$$
H\left( y-M^{-1/\alpha} N(y) \right) = \inf_X H = 1-\delta \quad \text{for every} \quad y \in C.
$$
Let us define
$$
A:= \overline{\co}(C \cup \lbrace y-M^{-1/\alpha} N(y) \: : \: y\in C \rbrace), \qquad F := H + d_A^{1+\alpha} \quad \text{on} \quad X,
$$
where $d_A$ stands for the distance function to $A.$ The function $d_A^{1+\alpha}$ is convex because $A$ is a convex subset. Given $x,h\in X$ we can find $y \in A$ is such that $d(x,A)=\| x-y\|$ because $X$ is reflexive. Then the inequality \eqref{inequalityrenorming} for $x-y$ and $h$ gives
$$
d_A(x+h)^{1+\alpha} + d_A(x-h)^{1+\alpha}-2d_A(x)^{1+\alpha} \leq \| x+h-y\|^{1+\alpha} + \| x-h-y\|^{1+\alpha} - 2  \| x-y\|^{1+\alpha} \leq L \|h\|^{1+\alpha}.
$$
Therefore, since $d_A^{1+\alpha}$ is continuous and convex, $d_A^{1+\alpha}$ is of class $C^{1,\alpha}(X);$ see \cite[Proposition 5.4]{AzagraMudarraExplicitFormulas}. Now it is enough to define $V=F^{-1} (-\infty,1] $ and imitate the proof of Theorem \ref{main theorem for C1omega Hilbert}.

\section*{Acknowledgements}
D. Azagra and C. Mudarra were partially supported by Grant MTM2015-65825-P and by the Severo Ochoa Program for Centres of Excellence in R\&D (Grant SEV-2015-0554).


\end{document}